\documentclass[10pt]{amsart}
\usepackage{amsmath,amscd}
\usepackage{amssymb}

\usepackage[all]{xy}

\usepackage[T1]{fontenc}

\usepackage{mathptmx}
\usepackage[scaled=.90]{helvet}
\usepackage{courier}

\author{Liran Shaul}
\address{Universiteit Antwerpen, Departement Wiskunde-Informatica, Middelheim campus,
Middelheimlaan 1,
2020 Antwerp, Belgium}
\email{Liran.Shaul@uantwerpen.be}
\newtheorem{thm}[equation]{Theorem}
\newtheorem{cor}[equation]{Corollary}
\newtheorem{prop}[equation]{Proposition}
\newtheorem{lem}[equation]{Lemma}
\theoremstyle{definition}

\newtheorem{rem}[equation]{Remark}

\newcommand{\surj}{\twoheadrightarrow}

\newcommand{\opn}{\operatorname}

\newcommand{\mfrak}[1]{\mathfrak{#1}}

\newcommand{\mrm}[1]{\mathrm{#1}}

\renewcommand{\k}{\Bbbk}

\renewcommand{\a}{\mfrak{a}}
\renewcommand{\b}{\mfrak{b}}

\numberwithin{equation}{section} 

\begin{document}

\title{Tensor product of dualizing complexes over a field}
\begin{abstract}
Let $\k$ be a field, and let $X,Y$ be two locally noetherian $\k$-schemes (respectively $\k$-formal schemes) with dualizing complexes $R_X$ and $R_Y$ respectively. We show that $R_X \boxtimes_{\k} R_Y$ (respectively its derived completion) is a dualizing complex over $X\times_{\k} Y$ if and only if $X\times_{\k} Y$ is locally noetherian of finite Krull dimension.
\end{abstract}

\thanks{The author acknowledges the support of the European Union for the ERC grant No 257004-HHNcdMir.}

\maketitle

\setcounter{tocdepth}{1}

\tableofcontents

\setcounter{section}{-1}
\section{Introduction}

Throughout this note, rings are assumed to be commutative and unital.  Given a ring $A$, we denote by $\mrm{D}(\opn{Mod} A)$ the derived category of $A$-modules, and by $\mrm{D}^{\mrm{b}}(\opn{Mod} A)$ and $\mrm{D}^{\mrm{b}}_{\mrm{f}}(\opn{Mod} A)$ its triangulated subcategories made of bounded complexes, and bounded complexes with coherent cohomology respectively. We will also use commutative DG-algebras. Given such a DG-algebra $A$, we will denote the category of differential graded $A$-modules by $\opn{DGMod} A$, and its derived category by $\mrm{D}(\opn{DGMod} A)$.

Dualizing complexes, first introduced in \cite{RD} half a century ago, are now a ubiquitous tool in commutative algebra and algebraic geometry. In this note we are concerned with dualizing complexes over a fibre product of schemes or formal schemes over a field. 

In the first section we work with ordinary schemes. Our main result in Section 1 shows that if two locally noetherian schemes $X,Y$, over a field $\k$, have dualizing complexes $R_X,R_Y$, then the only obstruction for $X\times_{\k} Y$ to possess a dualizing complex is the trivial one, namely,  $X\times_{\k} Y$ must be locally noetherian, and of finite Krull dimension. In that case we show that the box tensor product $R_X \boxtimes_{\k} R_Y$ is a dualizing complex over $X\times_{\k} Y$. This is proven in Corollary \ref{cor:main} below. If the schemes involved are of finite type over $\k$, then this is not new, and could be easily deduced from the results of \cite{RD}. In fact, in that case one can even replace $\k$ by a Gorenstein ring, assume one of $X,Y$ is flat over it, and replace tensor product with derived tensor product. We however make no finiteness assumption on either of the maps $X\to \k,Y \to \k$. 

One interesting consequence of this result in the affine case, given in Corollary \ref{cor:inj-dim} below, is the fact that for such noetherian rings, the tensor product functor $-\otimes_{\k} -:\mrm{D}^{\mrm{b}}_{\mrm{f}}(\opn{Mod} A) \times \mrm{D}^{\mrm{b}}_{\mrm{f}}(\opn{Mod} B) \to \mrm{D}^{\mrm{b}}_{\mrm{f}}(\opn{Mod} A\otimes_{\k} B)$ preserves finite injective dimension.

In section 2 we switch to the more difficult case of formal schemes. We are able to reproduce the above result in the formal case, and prove that if $\mathfrak{X}$ and $\mathfrak{Y}$ are two locally noetherian formal schemes over a field $\k$, with dualizing complex $R_{\mathfrak{X}}$ and $R_{\mathfrak{Y}}$ respectively, and if $\mathfrak{X}\times_{\k} \mathfrak{Y}$ is locally noetherian and of finite Krull dimension, then the derived completion and derived torsion of $R_{\mathcal{X}} \boxtimes_{\k} R_{\mathcal{Y}}$ are c-dualizing and t-dualizing complexes (notions that are recalled in Section 2 below) over $\mathfrak{X}\times_{\k} \mathfrak{Y}$. This is given in Theorem \ref{thm:tensor-dual-adic} below.	

To understand why the formal case is much more involved, consider the simplest corresponding affine situation, where $\k$ is a field, and $A$ and $B$ are two noetherian Gorenstein $\k$-algebras of finite Krull dimension, which are adically complete with respect to some ideals $\a \subseteq A$ and $\b \subseteq B$. The Gorenstein hypothesis implies that $A$ and $B$ are dualizing complexes over themselves, so what we need to prove is that in this situation, the completed tensor product $A\widehat{\otimes_{\k}} B$ is also a Gorenstein ring, whenever it is noetherian of finite Krull dimension. However, the ring $A\otimes_{\k} B$ is usually non-noetherian, so we do not know if the completion map $A\otimes_{\k} B \to A\widehat{\otimes_{\k}} B$ is flat, and so we do not know if in general the maps $A \to A\widehat{\otimes_{\k}} B$ and $B \to A\widehat{\otimes_{\k}} B$ are flat. This rules out attempts to prove such a result using the methods used in the corresponding discrete case (i.e, when $\a=0$ and $\b=0$, so that $A\otimes_{\k} B$ is noetherian) given in \cite{TY} and other similar papers. As a replacement for flatness, we rely heavily on the theory of weakly proregular ideals of \cite{AJL1,Sc,PSY1}. Using it and some other homological and homotopical tools, we are able to prove the above mentioned result about dualizing complexes over fiber product of formal schemes.

\textbf{Acknowledgments.} The author would like to thank Amnon Yekutieli for some useful
suggestions.

\section{Tensor product of dualizing complexes over ordinary schemes}

We shall need the following result, which is contained in the proof of \cite[Corollary 1.4]{Ka}.

\begin{prop}\label{prop:Gorfinite}
Let $A$ be a commutative noetherian ring. Assume that $A$ has a dualizing complex. Then there is a finite type $A$-algebra $A'$ which is Gorenstein of finite Krull dimension, and such that there is a surjection $A' \surj A$.
\end{prop}
\begin{proof}
By \cite[Corollary 1.4]{Ka}, there is a Gorenstein ring $A'$ of finite Krull dimension, and a surjection $A'\surj A$, so it is enough to verify that this ring is a finitely generated $A$-algebra. The first step in the proof of \cite[Corollary 1.4]{Ka} reduces to the case where the codimension function of $A$ is constant on its associated primes. One way to do this is using \cite[Lemma 5.5]{Ka}, which says that if a ring $A$ is noetherian, universally catenary, and has a codimension function (all these properties are satisfied by a ring possessing a dualizing complex), then there exist a finite type $A$-algebra $B$ whose codimension function is constant on its associated primes, and such that there is a surjection $B \surj A$. Hence, we may assume without loss of generality that the codimension function of our $A$ is constant on its associated primes. Next, for such an $A$, it is shown in \cite{Ka}, that there some ideal $I\subseteq A$ of positive height, such that the Rees algebra $A(I) = \oplus_{n\ge 0} (Ix)^n \subseteq A[x]$ is Cohen-Macaulay. Note that if $I=(f_1,\dots,f_m)$, then the map $A[x_1,\dots,x_m] \to A(I)$ given by $x_i \mapsto f_i\cdot x$ is surjective, so that $A(I)$ is of finite type over $A$. As $A$ is a quotient of $A(I)$, we reduce to the case where $A$ is Cohen-Macaulay. Then, in the final step of \cite[Corollary 1.4]{Ka}, it is observed that by \cite[Theorem 4.3]{Sh}, there is some finitely generated $A$-module $M$, such that the trivial extension ring $A' := A\times M$ (in the sense of \cite[Chapter 25, page 191]{Ma}) is a finite dimensional Gorenstein ring. As there is a surjection $A' \to A$, and as $A'$ is clearly a finite type $A$-algebra, we are done.
\end{proof}

\begin{lem}\label{lem:canonical-map}
Let $\k$ be a field, and let $A$ and $B$ be two noetherian $\k$-algebras with dualizing complexes $R_A$ and $R_B$ respectively, such that $A\otimes_{\k} B$ is a noetherian ring. Then the canonical map
\[
A\otimes_{\k} B \to \mrm{R}\opn{Hom}_{A\otimes_{\k} B}(R_A \otimes_{\k} R_B,R_A\otimes_{\k} R_B)
\]
is an isomorphism in $\mrm{D}(\opn{Mod} A\otimes_{\k} B)$.
\end{lem}
\begin{proof}
Since $R_A$ and $R_B$ are dualizing complexes, they have finitely generated bounded cohomologies, and in particular we may assume that they are bounded. Hence, by \cite[Lemma 8.4]{YZ}, there is an isomorphism
\[
\mrm{R}\opn{Hom}_A(R_A,R_A)\otimes_{\k} \mrm{R}\opn{Hom}_B(R_B,R_B) \cong \mrm{R}\opn{Hom}_{A\otimes_{\k} B} (R_A\otimes_{\k} R_B,R_A\otimes_{\k} R_B)
\]
in $\mrm{D}(\opn{Mod} A\otimes_{\k} B)$. Again, the fact that $R_A$ and $R_B$ are dualizing complexes implies that
\[
\mrm{R}\opn{Hom}_A(R_A,R_A)\otimes_{\k} \mrm{R}\opn{Hom}_B(R_B,R_B) \cong A\otimes_{\k} B.
\]
Composing these two isomorphisms, we deduce that there is some isomorphism
\[
\mrm{R}\opn{Hom}_{A\otimes_{\k} B}(R_A \otimes_{\k} R_B,R_A\otimes_{\k} R_B) \cong A\otimes_{\k} B
\]
in $\mrm{D}(\opn{Mod} A\otimes_{\k} B)$. Hence, by an unpublished result of Foxby, given in \cite[Proposition 2.3]{AIL}, the canonical map
\[
A\otimes_{\k} B \to \mrm{R}\opn{Hom}_{A\otimes_{\k} B}(R_A \otimes_{\k} R_B,R_A\otimes_{\k} R_B)
\]
is also an isomorphism.
\end{proof}

The following lemma is probably well known. We reproduce its easy proof for the convenience of the reader:
\begin{lem}\label{lem:inj-perfect}
Let $A$ be a noetherian ring, and let $R$ be a dualizing complex over $A$. A complex $M \in \mrm{D}^{\mrm{b}}_{\mrm{f}}(\opn{Mod} A)$ has finite injective dimension over $A$ if and only if the complex $\mrm{R}\opn{Hom}_A(M,R)$ is perfect.
\end{lem}
\begin{proof}
Since $\mrm{R}\opn{Hom}_A(M,R)$ has finitely generated cohomologies, by \cite[Corollary 2.10.F]{AF}, it is perfect if and only if it has finite flat dimension, and by \cite[Proposition V.2.6]{RD}, this happens if and only if
\[
\mrm{R}\opn{Hom}_A(\mrm{R}\opn{Hom}_A(M,R),R) \cong M
\]
has finite injective dimension over $A$.
\end{proof}

\begin{lem}\label{lem:one-implies-all}
Let $\k$ be a field, and let $A$ and $B$ be two noetherian $\k$-algebras, such that $A\otimes_{\k} B$ is also noetherian. Assume that there are dualizing complexes $R_A$ over $A$ and $R_B$ over $B$, such that $R_A \otimes_{\k} R_B$ is a dualizing complex over $A\otimes_{\k} B$. Then for any dualizing complex $S_A$ over $A$, and any dualizing complex $S_B$ over $B$, the complex $S_A \otimes_{\k} S_B$ is a dualizing complex over $A\otimes_{\k} B$.
\end{lem}
\begin{proof}
By Lemma \ref{lem:canonical-map}, the canonical map
\[
A\otimes_{\k} B \to \mrm{R}\opn{Hom}_{A\otimes_{\k} B}(S_A \otimes_{\k} S_B,S_A\otimes_{\k} S_B)
\]
is an isomorphism, so it is enough to show that $S_A\otimes_{\k} S_B$ has finite injective dimension over $A\otimes_{\k} B$. Since by assumption $R_A\otimes_{\k} R_B$ is a dualizing complex, by Lemma
\ref{lem:inj-perfect}, it is enough to show that the complex
\[
\mrm{R}\opn{Hom}_{A\otimes_{\k} B}(S_A,\otimes_{\k} S_B,R_A\otimes_{\k} R_B)
\]
is perfect, but this is clear, since by \cite[Lemma 8.4]{YZ}, it is isomorphic to
\[
\mrm{R}\opn{Hom}_A(S_A,R_A) \otimes_{\k} \mrm{R}\opn{Hom}_B(S_B,R_B),
\]
and because the (box) tensor product of two finitely generated projectives is a finitely generated projective.
\end{proof}

In the next lemma, we will have to use differential graded algebras, and dualizing DG-modules over them. We refer the reader to \cite{Ye2} for the terminology regarding DG-algebras used in this lemma. The lemma essentially says that if for a pair of $\k$-algebras $A',B'$, our main theorem about tensor product of dualizing complexes holds, then it also holds for any pair of quotients $A'\surj A, B'\surj B$. 

\begin{lem}\label{lem:tensors}
Let $\k$ be a field, and let $A',B'$ be two noetherian $\k$-algebras with dualizing complexes $R_A$ and $R_B$ respectively. Assume that $R_A \otimes_{\k} R_B$ is a dualizing complex over the noetherian ring $A'\otimes_{\k} B'$.
Let $A$ be an $A'$-algebra, and let $B$ be a $B'$-algebra such that the structure maps $A'\to A$ and $B'\to B$ are surjective. Then
\[
\mrm{R}\opn{Hom}_{A'}(A,R_A) \otimes_{\k} \mrm{R}\opn{Hom}_{B'}(B,R_B)
\]
is a dualizing complex over $A\otimes_{\k} B$.
\end{lem}
\begin{proof}
Since the map $A'\otimes_{\k} B' \to A\otimes_{\k} B$ is finite, it follows from \cite[Proposition V.2.4]{RD} that 
\[
\mrm{R}\opn{Hom}_{A'\otimes_{\k} B'}(A\otimes_{\k} B,R_A\otimes_{\k} R_B)
\]
is a dualizing complex over $A\otimes_{\k} B$. By \cite[Lemma 8.4]{YZ}, there is an isomorphism
\[
\mrm{R}\opn{Hom}_{A'}(A,R_A) \otimes_{\k} \mrm{R}\opn{Hom}_{B'}(B,R_B) \cong \mrm{R}\opn{Hom}_{A'\otimes_{\k} B'}(A\otimes_{\k} B,R_A\otimes_{\k} R_B)
\]
obtained by replacing $A$ and $B$ by projective resolutions over $A'$ and $B'$ respectively.
However, this isomorphism is only $A'\otimes_{\k} B'$-linear, and the author does not know how to show that there is such an $A\otimes_{\k} B$-linear isomorphism. Instead, using \cite[Proposition 2.2.8]{Av} let $A' \to \widetilde{A} \cong A$ and $B' \to \widetilde{B} \cong B$ be DG-algebra resolutions of $A'\to A$ and $B' \to B$ respectively, such that $\widetilde{A}^0=A'$ and $\widetilde{B}^0=B'$, for each $i<0$, $\widetilde{A}^i$ is a finitely generated projective $\widetilde{A}^0$-module, $\widetilde{B}^i$ is a finitely generated projective $\widetilde{B}^0$-module, and for each $i>0$, $\widetilde{A}^i = \widetilde{B}^i = 0$. Then, as shown in the proof of \cite[Lemma 8.4]{YZ}, the natural map
\[
\opn{Hom}_{A'}(\widetilde{A},R_A) \otimes_{\k} \opn{Hom}_{B'}(\widetilde{B},R_B) \to \opn{Hom}_{A'\otimes_{\k} B'}(\widetilde{A}\otimes_{\k} \widetilde{B},R_A\otimes_{\k} R_B)
\]
which is clearly $\widetilde{A}\otimes_{\k} \widetilde{B}$-linear, is an isomorphism. Since $\widetilde{A}$ is K-projective over $A'$, $\widetilde{B}$ is K-projective over $B'$, and $\widetilde{A}\otimes_{\k} \widetilde{B}$ is K-projective over $A'\otimes_{\k} B'$, we deduce that there is an isomorphism
\[
\mrm{R}\opn{Hom}_{A'}(\widetilde{A},R_A) \otimes_{\k} \mrm{R}\opn{Hom}_{B'}(\widetilde{B},R_B) \cong \mrm{R}\opn{Hom}_{A'\otimes_{\k} B'}(\widetilde{A}\otimes_{\k} \widetilde{B},R_A\otimes_{\k} R_B)
\]
in $\mrm{D}(\opn{DGMod} \widetilde{A}\otimes_{\k} \widetilde{B})$. By \cite[Proposition 7.5(1)]{Ye2}, the right hand side is a dualizing DG-module over $\widetilde{A}\otimes_{\k} \widetilde{B}$, so that, the left hand side is also a dualizing DG-module. As there are isomorphisms
\[
\mrm{R}\opn{Hom}_{A'}(\widetilde{A},R_A) \cong
\mrm{R}\opn{Hom}_{A'}(A,R_A)
\]
and
\[
\mrm{R}\opn{Hom}_{B'}(\widetilde{B},R_B) \cong 
\mrm{R}\opn{Hom}_{B'}(B,R_B)
\]
over $\widetilde{A}$ and $\widetilde{B}$ respectively, it follows that the DG-module
\[
\mrm{R}\opn{Hom}_{A'}(A,R_A) \otimes_{\k} \mrm{R}\opn{Hom}_{B'}(B,R_B)
\]
is a dualizing DG-module over $\widetilde{A}\otimes_{\k} \widetilde{B}$. Set $R = \mrm{R}\opn{Hom}_{A'}(A,R_A) \otimes_{\k} \mrm{R}\opn{Hom}_{B'}(B,R_B) \in \mrm{D}(\opn{Mod} A\otimes_{\k} B)$. Because $\widetilde{A}\otimes_{\k} \widetilde{B} \to A\otimes_{\k} B$ is a quasi-isomorphism, 
the fact that the image of $R$ in the derived category over $\widetilde{A}\otimes_{\k} \widetilde{B}$ has a finite injective dimension implies that $R$ has finite injective dimension over $A\otimes_{\k} B$. By Lemma \ref{lem:canonical-map}, the canonical map $A\otimes_{\k} B \to \mrm{R}\opn{Hom}_{A\otimes_{\k} B}(R,R)$ is an isomorphism.
Hence, $R = \mrm{R}\opn{Hom}_{A'}(A,R_A) \otimes_{\k} \mrm{R}\opn{Hom}_{B'}(B,R_B)$ is a dualizing complex over $A\otimes_{\k} B$.
\end{proof}

We now give the main result of this section.

\begin{thm}\label{thm:tensor-dual}
Let $\k$ be a field, and let $A,B$ be commutative noetherian $\k$-algebras. Assume that $A$ and $B$ have dualizing complexes. Then the ring $A\otimes_{\k} B$ has a dualizing complex if and only if $A\otimes_{\k} B$ is noetherian of finite Krull dimension. In that case, for every dualizing complexes $R_A$ over $A$ and $R_B$ over $B$, the complex $R_A \otimes_{\k} R_B$ is a dualizing complex over $A\otimes_{\k} B$.
\end{thm}
\begin{proof}
The only if part is well known. Assume that $A\otimes_{\k} B$ is noetherian of finite Krull dimension.
Let $A'\surj A$ and $B'\surj B$ be the Gorenstein rings guaranteed to exist from Proposition \ref{prop:Gorfinite}. Since $A\otimes_{\k} B$ is noetherian of finite Krull dimension, and since $A'$ (respectively $B'$) is a finite type $A$ (resp. $B$)-algebra, it follows that $A'\otimes_{\k} B'$ is also noetherian of finite Krull dimension. Hence, by \cite[Theorem 6(a)]{TY}, the ring $A'\otimes_{\k} B'$ is also Gorenstein. Let $R := \mrm{R}\opn{Hom}_{A'}(A,A')$, and let $S := \mrm{R}\opn{Hom}_{B'}(B,B')$. As $A'$ is Gorenstein of finite Krull dimension, $A'$ is a dualizing complex over $A'$, so that by \cite[Proposition V.2.4]{RD} $R$ is a dualizing complex over $A$, and in the same manner, $S$ is a dualizing complex over $B$. Similarly, since $A'\otimes_{\k} B'$ is Gorenstein of finite Krull dimension, $A'\otimes_{\k} B'$ is a dualizing complex over $A'\otimes_{\k} B'$. Thus, the conditions of Lemma \ref{lem:tensors} are satisfied for $A'\to A$ and $B'\to B$, so that
\[
\mrm{R}\opn{Hom}_{A'}(A,A')\otimes_{\k} \mrm{R}\opn{Hom}_{B'}(B,B') = R \otimes_{\k} S
\]
is a dualizing complex over $A\otimes_{\k} B$. Hence, by Lemma \ref{lem:one-implies-all}, the same is true for the complex $R_A \otimes_{\k} R_B$.
\end{proof}

Since one can check the property of being a dualizing complex on an affine open cover (because of \cite[Lemma II.7.16]{RD}), we obtain:

\begin{cor}\label{cor:main}
Let $\k$ be a field, and let $X,Y$ be two locally noetherian $\k$-schemes with dualizing complexes $R_X$ and $R_Y$ respectively. If $X\times_{\k} Y$ is locally noetherian of finite Krull dimension, then $R_X \boxtimes_{\k} R_Y$ is a dualizing complex over $X\times_{\k} Y$.
\end{cor}

\begin{cor}\label{cor:inj-dim}
Let $\k$ be a field, and let $A,B$ be two $\k$-algebras. Assume that $A$ and $B$ have dualizing complexes $R_A$ and $R_B$ respectively, and that $A\otimes_{\k} B$ is noetherian of finite Krull dimension. Given a complex $M\in \mrm{D}^{\mrm{b}}_{\mrm{f}}(\opn{Mod} A)$ and a complex $N\in \mrm{D}^{\mrm{b}}_{\mrm{f}}(\opn{Mod} B)$, if $M$ has finite injective dimension over $A$, and $N$ has finite injective dimension over $B$, then $M\otimes_{\k} N$ has finite injective dimension over $A\otimes_{\k} B$.
\end{cor}
\begin{proof}
By Lemma \ref{lem:inj-perfect}, the complexes $\mrm{R}\opn{Hom}_A(M,R_A)$ and $\mrm{R}\opn{Hom}_B(N,R_B)$ are perfect over $A$ and $B$ respectively. Hence, the complex $\mrm{R}\opn{Hom}_A(M,R_A)\otimes_{\k}\mrm{R}\opn{Hom}_B(N,R_B)$ is perfect over $A\otimes_{\k} B$. By \cite[Lemma 8.4]{YZ}, there is an isomorphism 
\[
\mrm{R}\opn{Hom}_A(M,R_A)\otimes_{\k}\mrm{R}\opn{Hom}_B(N,R_B) \cong \mrm{R}\opn{Hom}_{A\otimes_{\k} B}(M\otimes_{\k} N,R_A\otimes_{\k} R_B),
\]
and since by Theorem \ref{thm:tensor-dual}, $R_A\otimes_{\k} R_B$ is a dualizing complex over $A\otimes_{\k} B$, we have that 
\[
M \otimes_{\k} N \cong \mrm{R}\opn{Hom}_{A\otimes_{\k} B}(\mrm{R}\opn{Hom}_{A\otimes_{\k} B}(M\otimes_{\k} N,R_
A\otimes_{\k} R_B),R_A\otimes_{\k} R_B)
\]
so the result follows from applying Lemma \ref{lem:inj-perfect} again.
\end{proof}

\begin{rem}
The fact that Corollary \ref{cor:inj-dim} follows from the theorem about tensor product of dualizing complexes was already observed in \cite[Corollary 8.6]{YZ}, in a noncommutative situation. The result given there, in the commutative setting, makes the assumption that both $A$ and $B$ are finitely generated $\k$-algebras.
\end{rem}

\section{Tensor product of dualizing complexes over formal schemes}

We now turn to generalize Theorem \ref{thm:tensor-dual} to formal schemes. To do that, we first recall some adic homological algebra. We refer the reader to \cite{AJL1,AJL2,PSY1,PSY2,Sc} for a detailed treatment of the material below. By a preadic ring $(A,\a)$, we shall mean a commutative ring $A$ equipped with an adic topology generated by some finitely generated ideal $\a \subseteq A$ (It is important to note that we do not assume that $A$ is noetherian). Given a preadic ring $(A,\a)$, there are functors $\Gamma_{\a} (-) := \varinjlim \opn{Hom}_A(A/{\a}^n,-) $ and $\Lambda_{\a}(-) := \varprojlim A/{\a}^n \otimes_A - $ called the $\a$-torsion and $\a$-completion functors. These are both additive functors $\opn{Mod} A \to \opn{Mod} A$. The $A$-module $\widehat{A} := \Lambda_{\a}(A)$ has a structure of a commutative ring, and there is a natural map $A\to \widehat{A}$. If this map is bijective then we will call $(A,\a)$ an adic ring, and say that $A$ is $\a$-adically complete. For any $M \in \opn{Mod} A$, the $A$-modules $\Gamma_{\a}(M)$ and $\Lambda_{\a}(M)$ naturally carry a $\widehat{A}$-module structure, so that we obtain functors $\widehat{\Gamma}_{\a}, \widehat{\Lambda}_{\a} : \opn{Mod} A \to \opn{Mod} \widehat{A}$ defined by exactly the same formulas as $\Gamma_{\a}$ and $\Lambda_{\a}$. 
The derived functors \[
\mrm{R}\Gamma_{\a}, \mrm{L}\Lambda_{\a} :\mrm{D}(\opn{Mod} A) \to \mrm{D}(\opn{Mod} A)
\]
and
\[
\mrm{R}\widehat{\Gamma}_{\a}, \mrm{L}\widehat{\Lambda}_{\a} : \mrm{D}(\opn{Mod} A) \to \mrm{D}(\opn{Mod} \widehat{A})
\]
exist. $\mrm{R}\Gamma_{\a}$ and $\mrm{R}\widehat{\Gamma}_{\a}$ are calculated using K-injective resolutions, while $\mrm{L}\Lambda_{\a}$ and $\mrm{L}\widehat{\Lambda}_{\a}$ are calculated using K-flat resolutions. See \cite[Section 1]{AJL1} for a proof.

For any $M \in \mrm{D}(\opn{Mod} A)$, there are canonical morphisms $\mrm{R}\Gamma_{\a}(M) \to M$ and $M \to \mrm{L}\Lambda_{\a} (M)$ in $\mrm{D}(\opn{Mod} A)$. If these maps are isomorphisms we say that $M$ is cohomologically $\a$-torsion and cohomologically $\a$-adically complete respectively. The collections of all cohomologically $\a$-torsion and cohomologically $\a$-adically complete complexes form two full triangulated subcategories of $\mrm{D}(\opn{Mod} A)$. These are denoted by $\mrm{D}(\opn{Mod} A)_{\a\opn{-tor}}$ and $\mrm{D}(\opn{Mod} A)_{\a\opn{-com}}$ respectively.

Given a ring $A$ and a finite sequence of elements $\mathbf{a}$, there is a bounded complex of free $A$-modules, $\opn{Tel}(A;\mathbf{a})$ called the telescope complex associated to $\mathbf{a}$. See \cite[Section 5]{PSY1} for its definition. If $A \to B$ is a ring homomorphism, and if $\mathbf{b}$ is the image of $\mathbf{a}$ under this map, then there is an isomorphism of complexes $\opn{Tel}(A;\mathbf{a}) \otimes_A B \to \opn{Tel}(B;\mathbf{b})$. Given an ideal $\a\subseteq A$, and a finite sequence of elements $\mathbf{a} \subseteq A$ that generates $\a$, there is a morphism of functors $\mrm{R}\Gamma_{\a}(-) \to \opn{Tel}(A;\mathbf{a}) \otimes_A -$. If this morphism is a quasi-isomorphism then $\mathbf{a}$ and $\a$ are said to be weakly proregular. See \cite[Section 2]{Sc}, and in particular \cite[Definition 2.3]{Sc}. In a noetherian ring, every ideal and every finite sequence are weakly proregular.
If the ideal $\a$ is weakly proregular, then the functors 
\begin{eqnarray}
\nonumber
\mrm{R}\Gamma_{\a}:\mrm{D}(\opn{Mod} A)_{\a\opn{-com}} \to \mrm{D}(\opn{Mod} A)_{\a\opn{-tor}},\\
\nonumber
\mrm{L}\Lambda_{\a}:\mrm{D}(\opn{Mod} A)_{\a\opn{-tor}} \to \mrm{D}(\opn{Mod} A)_{\a\opn{-com}}
\end{eqnarray}
 are quasi-inverse to each other, and induce an equivalence between these two triangulated categories, called the Matlis-Greenlees-May equivalence. If $A$ is noetherian and $\a$-adically complete, then 
\[
\mrm{D}^{\mrm{b}}_{\mrm{f}}(\opn{Mod} A) \subseteq \mrm{D}(\opn{Mod} A)_{\a\opn{-com}}.
\]

In that case, the essential image of $\mrm{D}^{\mrm{b}}_{\mrm{f}}(\opn{Mod} A)$ under the functor $\mrm{R}\Gamma_{\a}$ is denoted by $\mrm{D}^{\mrm{b}}(\opn{Mod} A)_{\a-\opn{cof}}$. This is a triangulated category, called the category of cohomologically $\a$-adically cofinite complexes, and is equivalent to the category $\mrm{D}^{\mrm{b}}_{\mrm{f}}(\opn{Mod} A)$. See \cite{PSY2} for a study of this category.

The following proposition, whose proof is immediate from the definitions will be useful in the sequel.

\begin{prop}\label{prop:QLambda}
Let $A$ be a commutative ring, let $\a \subseteq A$ be an ideal, and let $\widehat{A}:=\Lambda_{\a}(A)$. Let $Q:\mrm{D}(\opn{Mod} \widehat{A}) \to \mrm{D}(\opn{Mod} A)$ be the forgetful functor. Then there are isomorphisms
\[
Q \circ \mrm{L}\widehat{\Lambda}_{\a} (-) \cong \mrm{L}\Lambda_{\a} (-)
\]
and
\[
Q \circ \mrm{R}\widehat{\Gamma}_{\a} (-) \cong \mrm{R}\Gamma_{\a} (-)
\]
of functors $\mrm{D}(\opn{Mod} A) \to \mrm{D}(\opn{Mod} A)$.
\end{prop}

\subsection{A reduction lemma}

The aim of this subsection is to prove Lemma \ref{lem:reduction-to-c-dual}, which allows us to reduce  certain questions over $\widehat{A}$ to questions over $A$.

The next lemma was inspired by a result of Yekutieli (private communication).
\begin{lem}\label{lem:rhom-qis-dg}
Let $A\to B$ be a quasi-isomorphism of commutative DG-algebras. Let $C$ be a $B$-algebra, and let $Q:\mrm{D}(\opn{DGMod} B) \to \mrm{D}(\opn{DGMod} A)$ be the forgetful functor. Then there is an isomorphism
\[
\mrm{R}\opn{Hom}_B(C,-) \cong \mrm{R}\opn{Hom}_A(C,Q(-))
\]
of functors $\mrm{D}(\opn{DGMod} B) \to \mrm{D}(\opn{Mod} C)$.
\end{lem}
\begin{proof}
Let $M \in \mrm{D}(\opn{DGMod} B)$. Let $M \to I_B$ be a K-injective resolution of $M$ over $B$, and let $Q(M) \to I_A$ be a K-injective resolution of $M$ over $A$. The functor $Q$ induces an isomorphism $I_B \cong I_A$ in $\mrm{D}(\opn{DGMod} A)$. Since $I_A$ is K-injective, there is some $A$-linear quasi-isomorphism $\alpha:I_B \to I_A$. Composition with $\alpha$ induces a map $\alpha': \opn{Hom}_B(C,I_B) \to \opn{Hom}_A(C,I_A)$, and this map is clearly $C$-linear. It is enough to show that $\alpha'$ is a quasi-isomorphism. To see this, consider the map $\phi:I_B \to \opn{Hom}_A(B,I_A)$ given by $\phi(x)(b) = \alpha(b\cdot x)$, for $b \in B$ and $x \in I_B$. This map fits into a commutative diagram
\[
\xymatrixcolsep{5pc}
\xymatrix{
I_B \ar[r]^{\phi} \ar[d]^{\alpha} & \opn{Hom}_A(B,I_A) \ar[d]\\
I_A \ar[r] & \opn{Hom}_A(A,I_A)
}
\]
Because $A\to B$ is a quasi-isomorphism, and $I_A$ is K-injective, the right vertical map is a quasi-isomorphism. Hence, $\phi$ is also a quasi-isomorphism between two K-injective DG $B$-modules, so it is a homotopy equivalence. Hence, in the commutative diagram
\[
\xymatrixcolsep{5pc}
\xymatrix{
\opn{Hom}_B(C,I_B) \ar[r]^{\alpha'}\ar[d]^{\opn{Hom}_B(1_C,\phi)} & \opn{Hom}_A(C,I_A)\ar[d] \\
\opn{Hom}_B(C,\opn{Hom}_A(B,I_A)) \ar[r] & \opn{Hom}_A(C\otimes_B B, I_A)
}
\]
the left vertical arrow induced by this homotopy equivalence is a quasi-isomorphism, while  the right vertical arrow and the bottom horizontal arrow are obviously isomorphisms. Hence, $\alpha'$ is a quasi-isomorphism, as claimed.
\end{proof}

\begin{lem}\label{lem:tor-calculation}
Let $A$ be a commutative ring, let $\a\subseteq A$ be a finitely generated weakly proregular ideal, and set $\widehat{A} :=\Lambda_{\a}(A)$. Let $J\subseteq A$ be an ideal , and assume that there are integers $m,n$, such that $\a^m \subseteq J \subseteq \a^n$. Then for all $i \ne 0$, we have that 
\[
\opn{Tor}^A_i(A/J,\widehat{A}) = 0.
\]
\end{lem}
\begin{proof}
Let $B = A/J$. Let $\mathbf{a}$ be a finite sequence of elements that generates $\a$, and let $\mathbf{b}$ be its image in $B$. Note that by assumption, each element of $\mathbf{b}$ is nilpotent. Hence, by \cite[Lemma 7.4]{PSY1}, there is a $B$-linear homotopy equivalence $\opn{Tel}(B;\mathbf{b}) \to B$. By the base change property of the telescope complex, we deduce that there is an $A$-linear homotopy equivalence $\opn{Tel}(A;\mathbf{a}) \otimes_A B \to B$.

Since $\a$ is weakly proregular, by the Greenlees-May duality (specifically, by \cite[Item (iv) of Corollary after Theorem (0.3)*]{AJL1}, or by \cite[Lemma 7.6]{PSY1}), there is an isomorphism $ \widehat{A} \otimes_A \opn{Tel}(A;\mathbf{a})  \cong \opn{Tel}(A;\mathbf{a})$ in $\mrm{D}(\opn{Mod} A)$.

Combining these two isomorphisms, and the fact that $\opn{Tel}(A;\mathbf{a})$ is a K-flat complex, we obtain the following sequence of isomorphisms in $\mrm{D}(\opn{Mod} A)$:
\[
B \cong \opn{Tel}(A;\mathbf{a}) \otimes^{\mrm{L}}_A B \cong (\widehat{A}\otimes^{\mrm{L}}_A \opn{Tel}(A;\mathbf{a}) ) \otimes^{\mrm{L}}_A B \cong \widehat{A}\otimes^{\mrm{L}}_A  (\opn{Tel}(A;\mathbf{a})  \otimes^{\mrm{L}}_A B ) \cong \widehat{A} \otimes^{\mrm{L}}_A B.
\]
As $B$ is a complex concentrated in degree $0$, the result follows.
\end{proof}

\begin{lem}\label{lem:main-reduction}
Let $A$ be a commutative ring, let $\a\subseteq A$ be a finitely generated weakly proregular ideal, and set $\widehat{A} :=\Lambda_{\a}(A)$. Let $J\subseteq A$ be an ideal, and assume that there are integers $m,n$, such that $\a^m \subseteq J \subseteq \a^n$. Set $B=A/J$, and let $Q_A : \mrm{D}(\opn{Mod} \widehat{A}) \to \mrm{D}(\opn{Mod} A)$ be the forgetful functor. Then there is an isomorphism
\[
\mrm{R}\opn{Hom}_{\widehat{A}}(B,-) \cong \mrm{R}\opn{Hom}_A(B,Q_A(-))
\]
of functors $\mrm{D}(\opn{Mod} \widehat{A}) \to \mrm{D}(\opn{Mod} B)$.
\end{lem}
\begin{proof}
Let $A\to \widetilde{A} \cong \widehat{A}$ be a commutative semi-free DG-algebra resolution of $A\to \widehat{A}$, and let $Q_{\widetilde{A}}:\mrm{D}(\opn{Mod} \widehat{A}) \to \mrm{D}(\opn{DGMod} \widetilde{A})$ be the corresponding forgetful functor. Given $M \in \mrm{D}(\opn{Mod} \widehat{A})$, according to Lemma \ref{lem:rhom-qis-dg}, there is an isomorphism of functors
\[
\mrm{R}\opn{Hom}_{\widehat{A}}(B,M) \cong \mrm{R}\opn{Hom}_{\widetilde{A}}(B,Q_{\widetilde{A}}(M)).
\]
Let $Q_{\widetilde{A}}(M) \to I$ be a K-injective resolution of $Q_{\widetilde{A}}(M)$ over $\widetilde{A}$. Then there is an obvious $B$-linear isomorphism
\[
\mrm{R}\opn{Hom}_{\widetilde{A}}(B,Q_{\widetilde{A}}(M)) \cong \opn{Hom}_{\widetilde{A}}(B,I).
\]
According to  lemma \ref{lem:tor-calculation}, we have that $\opn{Tor}^A_i(B,\widehat{A}) = 0$ for all $i\ne 0$. Hence, the map $B \otimes_A \widetilde{A} \to B\otimes_A \widehat{A}$ induced by the map $\widetilde{A} \to \widehat{A}$ is a quasi-isomorphism.

Since ${\a}^m \subseteq J$, we have that $B \otimes_A A/{\a}^m \cong B$. On the other hand, since $\a$ is finitely generated, we have that $A/{\a}^m \otimes_A \widehat{A} \cong A/{\a}^m$. Combining these two facts, we deduce that $B\otimes_A \widehat{A} \cong B$. It follows that there is a quasi-isomorphism $B\otimes_A \widetilde{A} \to B$, which is $B$-linear on the left, and $\widetilde{A}$-linear on the right.\footnote{The main reason we needed the to take the DG-algebra resolution $A \to \widetilde{A} \cong \widehat{A}$ was in order to get these linearity conditions on this quasi-isomorphism. These allow us now to use the hom-tensor adjunction. The fact that there is such an $A$-linear isomorphism is already proved in Lemma \ref{lem:tor-calculation}, but this fact is not enough to use adjunction in the next step of the proof.}

This in turn induces a quasi-isomorphism
\[
\opn{Hom}_{\widetilde{A}}(B,I) \to \opn{Hom}_{\widetilde{A}}(B\otimes_A \widetilde{A},I),
\]
which by the hom-tensor adjunction is naturally isomorphic to
\[
\opn{Hom}_A(B,I).
\]

Since $A\to \widetilde{A}$ is flat, we deduce that $I$ is K-injective over $A$, so that $\opn{Hom}_A(B,I) \cong \mrm{R}\opn{Hom}_A(B,Q_A(M))$, which proves the claim.
\end{proof}

\newpage

\begin{lem}\label{lem:reduction-to-c-dual}
Let $A$ be a commutative ring, let $\a\subseteq A$ be a finitely generated weakly proregular ideal, and set $\widehat{A} :=\Lambda_{\a}(A)$. Let $J\subseteq A$ be an ideal, and assume that there are integers $m,n$, such that $\a^m \subseteq J \subseteq \a^n$. Set $B=A/J$. Then there are isomorphisms
\[
\mrm{R}\opn{Hom}_{\widehat{A}}(B,\mrm{R}\widehat{\Gamma}_{\a}(-)) \cong 
\mrm{R}\opn{Hom}_{\widehat{A}}(B,\mrm{L}\widehat{\Lambda}_{\a}(-)) \cong 
\mrm{R}\opn{Hom}_A(B,-)
\]
of functors $\mrm{D}(\opn{Mod} A) \to \mrm{D}(\opn{Mod}(B)$.
\end{lem}
\begin{proof}
Let $Q:\mrm{D}(\opn{Mod} \widehat{A}) \to \mrm{D}(\opn{Mod} A)$ be the forgetful functor.
According to Lemma \ref{lem:main-reduction}, there are $B$-linear isomorphisms of functors
\[
\mrm{R}\opn{Hom}_{\widehat{A}}(B,\mrm{R}\widehat{\Gamma}_{\a}(-)) \cong
\mrm{R}\opn{Hom}_{A}(B,Q(\mrm{R}\widehat{\Gamma}_{\a}(-)))
\]
and
\[
\mrm{R}\opn{Hom}_{\widehat{A}}(B,\mrm{L}\widehat{\Lambda}_{\a}(-)) \cong 
\mrm{R}\opn{Hom}_{A}(B,Q(\mrm{L}\widehat{\Lambda}_{\a}(-))) 
\]
By Proposition \ref{prop:QLambda}, these are isomorphic in $\mrm{D}(\opn{Mod} B)$ to
\[
\mrm{R}\opn{Hom}_{A}(B,\mrm{R}\Gamma_{\a}(-) )
\]
and
\[
\mrm{R}\opn{Hom}_{A}(B,\mrm{L}\Lambda_{\a}(-) )
\]
respectively. In the proof of Lemma \ref{lem:tor-calculation}, we have seen that $B \cong \opn{Tel}(A;\mathbf{a}) \otimes_A B$, which implies that $B$ is cohomologically $\a$-torsion. Hence, by the Greenlees-May duality (\cite[Theorem 0.3]{AJL1}, or \cite[Theorem 7.12]{PSY1}, there are natural isomorphisms
\[
\mrm{R}\opn{Hom}_{A}(B,\mrm{R}\Gamma_{\a}(-) ) \cong \mrm{R}\opn{Hom}_{A}(B,- ) \cong 
\mrm{R}\opn{Hom}_{A}(B,\mrm{L}\Lambda_{\a}(-) ).
\]
The isomorphisms constructed in \cite{AJL1,PSY1} are $A$-linear, but it is easy to verify that in our situation they are actually $B$-linear. This proves the claim.
\end{proof}

\subsection{The box tensor products over affine formal schemes}
 
 Next, we obtain some general finiteness results about the adic box tensor products. For a moment, we drop the assumption that $\k$ is a field, as it does not produce additional difficulties, and it seems that this result might be of independent interest in this greater generality.

\begin{prop}\label{prop:adic-box}
Let $\k$ be a commutative ring, and let $(A,\a)$ and $(B,\b)$ be two noetherian adic rings which are flat $\k$-algebras. Let $I = \a\otimes_{\k} B + A\otimes_{\k} \b$ be the ideal of definition of the adic topology on $A\otimes_{\k} B$, let $\widehat{I}$ be the ideal generated by its image in $\widehat{A\otimes_{\k} B}$, and assume that $I$ is weakly proregular (if $\k$ is a field this always holds), and that $\widehat{A\otimes_{\k} B}$ is noetherian.
Given $M \in \mrm{D}^{\mrm{b}}_{\mrm{f}}(\opn{Mod} A)$ and $N \in \mrm{D}^{\mrm{b}}_{\mrm{f}}(\opn{Mod} B)$ with $M$ having finite flat dimension over $\k$, we have that
\[
\mrm{L}\widehat{\Lambda}_{I} (M\otimes^{\mrm{L}}_{\k} N)  \in \mrm{D}^{\mrm{b}}_{\mrm{f}}(\opn{Mod} \widehat{A\otimes_{\k} B}),
\] 
and
\[
\mrm{R}\widehat{\Gamma}_{I} (M\otimes^{\mrm{L}}_{\k} N)  \in \mrm{D}^{\mrm{b}}(\opn{Mod} \widehat{A\otimes_{\k} B})_{\widehat{I}-\opn{cof}}.
\] 
\end{prop}
\begin{proof}
We first show that both of these complexes have bounded cohomology. Let $Q:\mrm{D}(\opn{Mod} \widehat{A\otimes_{\k} B}) \to \mrm{D}(\opn{Mod} A\otimes_{\k} B)$ be the forgetful functor. Clearly, a complex $X$ has bounded cohomology if and only if the complex $Q(X)$ has bounded cohomology. In view of Proposition \ref{prop:QLambda}, it is enough to show that the complexes
\[
\mrm{L}\Lambda_{I} (M\otimes^{\mrm{L}}_{\k} N)
\]
and
\[
\mrm{R}\Gamma_{I} (M\otimes^{\mrm{L}}_{\k} N)
\]
have bounded cohomology, but this follows immediately from the flat dimension assumption on $M$, combined with the fact that when $I$ is weakly proregular, the functors $\mrm{L}\Lambda_I$ and $\mrm{R}\Gamma_I$ have finite cohomological dimension (for example, by \cite[Corollary 4.28]{PSY1} and \cite[Corollary 5.27]{PSY1}).

Next, we show the claims about finiteness of the cohomologies. Let $P\to M$ and $Q\to N$ be bounded above resolutions made of finitely generated free modules. As $A$ is flat over $\k$, $P$ is also flat over $\k$, so that $M\otimes^{\mrm{L}}_{\k} N \cong P\otimes_{\k} Q$, and the latter is also a bounded above complex made of finitely generated free modules, so that 
\[
\mrm{L} \widehat{\Lambda}_I (M\otimes^{\mrm{L}}_{\k} N ) \cong \Lambda_I(P\otimes_{\k} Q).
\]
Since the completion functor commutes with finite direct sums, it follows that $\Lambda_I(P\otimes_{\k} Q)$ is also a bounded above complex made of finitely generated free modules, which shows that the cohomologies of this complex are finitely generated over $\widehat{A\otimes_{\k} B}$. It remains to show that
\[
\mrm{R}\widehat{\Gamma}_{I} (M\otimes^{\mrm{L}}_{\k} N)  \in \mrm{D}^{\mrm{b}}(\opn{Mod} \widehat{A\otimes_{\k} B})_{\widehat{I}-\opn{cof}}.
\]
As we already established that this complex is bounded, and as it is clearly cohomologically $\widehat{I}$-torsion, by \cite[Theorem 3.10]{PSY2}, it is enough to show that the complex
\[
\mrm{R}\opn{Hom}_{\widehat{A\otimes_{\k} B}}(\widehat{A\otimes_{\k} B}/\widehat{I} , \mrm{R}\widehat{\Gamma}_{I} (M\otimes^{\mrm{L}}_{\k} N) )
\]
has finitely generated cohomologies. By Lemma \ref{lem:reduction-to-c-dual}, there is an isomorphism
\[
\mrm{R}\opn{Hom}_{\widehat{A\otimes_{\k} B}}(\widehat{A\otimes_{\k} B}/\widehat{I} , \mrm{R}\widehat{\Gamma}_{I} (M\otimes^{\mrm{L}}_{\k} N) ) \cong 
\mrm{R}\opn{Hom}_{\widehat{A\otimes_{\k} B}}(\widehat{A\otimes_{\k} B}/\widehat{I} , \mrm{L}\widehat{\Lambda}_{I} (M\otimes^{\mrm{L}}_{\k} N) ),
\]
so the result follows from the first claim in this proposition.
\end{proof}

\begin{rem}
One might wonder why in the above proof we had to invoke the rather difficult theorem of \cite{PSY2}, instead of deducing the finiteness condition in the torsion case directly from the identity $\mrm{R}\widehat{\Gamma}_I (-) \cong \mrm{R}\Gamma_{\widehat{I}} \circ \mrm{L}\widehat{\Lambda}_I (-) $. The reason for that is that we do not know if this identity holds when $A\otimes_{\k} B \to \widehat{A\otimes_{\k} B}$ is not flat.
\end{rem}

\subsection{Tensor product of dualizing complexes over formal schemes}

In this subsection we will prove Theorem \ref{thm:tensor-dual-adic}, the main result of this section. First, we recall the definitions of dualizing complexes over affine formal schemes. See \cite[Section 2.5]{AJL2} and \cite[Section 5]{Ye1} for details (keeping in mind \cite[Theorem 3.10]{PSY2}). Let $(A,\a)$ be an adic noetherian ring. A complex $R \in \mrm{D}(\opn{Mod} A)$ which has finite injective dimension over $A$, and such that the canonical map $A \to \mrm{R}\opn{Hom}_A(R,R)$ is an isomorphism is called a c-dualizing complex if $R \in \mrm{D}^{\mrm{b}}_{\mrm{f}}(\opn{Mod} A)$, and is called a t-dualizing complex if $R \in \mrm{D}^{\mrm{b}}(\opn{Mod} A)_{\a-\opn{cof}}$.

The next lemma allows us to reduce the problem of determining if a complex over the completed tensor product is dualizing to a problem over discrete rings. We will then use Theorem \ref{thm:tensor-dual} to obtain the required result.

\begin{lem}\label{lem:reduction-to-discrete}
Let $\k$ be a field, and let $(A,\a)$ and $(B,\b)$ be two noetherian adic rings which are $\k$-algebras, such that the completed tensor product $\widehat{A\otimes_{\k} B}$ is noetherian of finite Krull dimension. Let $I$ be the ideal of definition of the adic topology on $\widehat{A\otimes_{\k} B}$, and let $M \in \mrm{D}^{\mrm{b}}_{\mrm{f}}(\opn{Mod} \widehat{A\otimes_{\k} B})$ (respectively $M \in \mrm{D}(\opn{Mod} \widehat{A\otimes_{\k} B})_{\opn{I-cof}})$. Then $M$ is a c-dualizing (resp. t-dualizing) complex over $\widehat{A\otimes_{\k} B}$ if and only if for each $n>0$ the complex
\[
\mrm{R}\opn{Hom}_{\widehat{A\otimes_{\k} B}}(A/{\a}^n \otimes_{\k} B/{\b}^n, M) \in \mrm{D}(\opn{Mod} A/{\a}^n \otimes_{\k} B/{\b}^n)
\]
is a dualizing complex over $A/{\a}^n \otimes_{\k} B/{\b}^n$.
\end{lem}
\begin{proof}
Consider the sequence of ideals $J_n = \ker(\widehat{A\otimes_{\k} B} \to A/{\a}^n \otimes_{\k} B/{\b}^n)$.  For every $n \in \mathbb{N}$, there is some $m \in \mathbb{N}$, such that $J_n \subseteq I^m$, and likewise, for every $n\in \mathbb{N}$, there is some $m\in \mathbb{N}$, such that $I^n \subseteq J_m$. Hence, 
\[
\varprojlim (\widehat{A\otimes_{\k} B}/J_n) \cong \widehat{A\otimes_{\k} B},
\]
and moreover, the two functors $\Gamma_I(-)$ and $\varinjlim \opn{Hom}_{\widehat{A\otimes_{\k} B}}( \widehat{A\otimes_{\k} B}/J_n,-)$ are canonically isomorphic.
With these observations, the result now follows from the proof of \cite[Lemma 2.5.10]{AJL2} (see also \cite[Satz 2]{Fa}).
\end{proof}

We now arrive to the main result of this section, an adic generalization of Theorem \ref{thm:tensor-dual}.
\begin{thm}\label{thm:tensor-dual-adic}
Let $\k$ be a field, and let $(A,\a)$ and $(B,\b)$ be two noetherian adic rings which are $\k$-algebras. Let $I$ be the ideal of definition of the adic topology on $A\otimes_{\k} B$.
Let $R_A$ be a c-dualizing complex over $(A,\a)$, and let $R_B$ be a c-dualizing complex over $(B,\b)$. Then the ring $\widehat{A\otimes_{\k} B}$ has dualizing complexes if and only if it is noetherian of finite Krull dimension. In that case, 
$\mrm{L}\widehat{\Lambda}_I(R_A\otimes_{\k} R_B)$ is a c-dualizing complex over $\widehat{A\otimes_{\k} B}$, and $\mrm{R}\widehat{\Gamma}_I(R_A\otimes_{\k} R_B)$ is a t-dualizing complex over $\widehat{A\otimes_{\k} B}$.
\end{thm}
\begin{proof}
Let $\widehat{I}$ be the ideal generated by the image of $I$ in $\widehat{A\otimes_{\k} B}$. According to \cite[Example 4.35]{PSY1}, the ideal $I$ is weakly proregular. Hence, by Proposition \ref{prop:adic-box}, we have that
\[
\mrm{L}\widehat{\Lambda}_I(R_A\otimes_{\k} R_B)  \in \mrm{D}^{\mrm{b}}_{\mrm{f}}(\opn{Mod} \widehat{A\otimes_{\k} B}),
\]
and
\[
\mrm{R}\widehat{\Gamma}_I(R_A\otimes_{\k} R_B) \in \mrm{D}^{\mrm{b}}(\opn{Mod} \widehat{A\otimes_{\k} B})_{\widehat{I}-\opn{cof}}
\]
By Lemma \ref{lem:reduction-to-discrete}, it is enough to show that for all $n$, the complexes
\[
\mrm{R}\opn{Hom}_{\widehat{A\otimes_{\k} B}}(A/{\a}^n \otimes_{\k} B/{\b}^n,\mrm{L}\widehat{\Lambda}_I(R_A\otimes_{\k} R_B))
\]
and
\[
\mrm{R}\opn{Hom}_{\widehat{A\otimes_{\k} B}}(A/{\a}^n \otimes_{\k} B/{\b}^n,\mrm{R}\widehat{\Gamma}_I(R_A\otimes_{\k} R_B))
\]
are dualizing complexes over $A/{\a}^n \otimes_{\k} B/{\b}^n$. 
By Lemma \ref{lem:reduction-to-c-dual}, both of these complexes are isomorphic as objects in  $\mrm{D}(\opn{Mod} A/{\a}^n \otimes_{\k} B/{\b}^n)$, and moreover, both of them are isomorphic to the complex
\[
\mrm{R}\opn{Hom}_{A\otimes_{\k} B}(A/{\a}^n \otimes_{\k} B/{\b}^n,R_A\otimes_{\k} R_B).
\]
Note that as the maps $A\to A/{\a}^n$ and $B \to B/{\b}^n$ are finite, the complexes
\[
\mrm{R}\opn{Hom}_A(A/{\a}^n,R_A)
\]
and 
\[
\mrm{R}\opn{Hom}_B(B/{\b}^n,R_B)
\]
are dualizing complexes over $A/{\a}^n$ and $B/{\b}^n$ respectively. Since the ring $A/{\a}^n \otimes_{\k} B/{\b}^n$ is noetherian of finite Krull dimension (being a quotient of the noetherian ring of finite Krull dimension $\widehat{A\otimes_{\k} B}$), it follows from Theorem \ref{thm:tensor-dual} that
\[
\mrm{R}\opn{Hom}_A(A/{\a}^n,R_A) \otimes_{\k} \mrm{R}\opn{Hom}_B(B/{\b}^n,R_B)
\]
is a dualizing complex over $A/{\a}^n \otimes_{\k} B/{\b}^n$. We now use the same trick as in the proof of Lemma  \ref{lem:tensors}. Thus, let $A \to \widetilde{A} \cong A/{\a}^n$ and $B \to \widetilde{B} \cong B/{\b}^n$ be DG-algebra resolutions of $A\to A/{\a}^n$ and $B\to B/{\b}^n$ respectively as in Lemma \ref{lem:tensors}. Then there is a $\widetilde{A}\otimes_{\k} \widetilde{B}$-linear isomorphism
\begin{eqnarray}
\nonumber
\mrm{R}\opn{Hom}_{A\otimes_{\k} B}(A/{\a}^n \otimes_{\k} B/{\b}^n,R_A\otimes_{\k} R_B) \cong\\
\nonumber
\mrm{R}\opn{Hom}_A(A/{\a}^n,R_A) \otimes_{\k} \mrm{R}\opn{Hom}_B(B/{\b}^n,R_B).
\end{eqnarray}
The right hand side is a dualizing complex over $A/{\a}^n \otimes_{\k} B/{\b}^n$, and hence, also a dualizing DG-module over $\widetilde{A}\otimes_{\k} \widetilde{B}$. Hence, the left hand side, which is a priori a complex over  $A/{\a}^n \otimes_{\k} B/{\b}^n$ is also a dualizing DG-module over $\widetilde{A}\otimes_{\k} \widetilde{B}$. Hence, by the argument used in the proof of Lemma \ref{lem:tensors}, we deduce that 
\[
\mrm{R}\opn{Hom}_{A\otimes_{\k} B}(A/{\a}^n \otimes_{\k} B/{\b}^n,R_A\otimes_{\k} R_B)
\]
is a dualizing complex over $A/{\a}^n \otimes_{\k} B/{\b}^n$, which establishes the theorem.
\end{proof}

Again, as in Corollary \ref{cor:main}, this generalizes immediately to formal schemes. As an immediate corollary, we obtain an adic generalization of \cite[Theorem 6(a)]{TY}. 
\begin{cor}
Let $\k$ be a field,  and let $(A,\a)$ and $(B,\b)$ be two adic noetherian Gorenstein $\k$-algebras of finite Krull dimension, such that $\widehat{A\otimes_{\k} B}$ is also noetherian of finite Krull dimension. Then $\widehat{A\otimes_{\k} B}$ is also a Gorenstein ring.
\end{cor}

\begin{rem}
As far as we know, all similar results in the literature concerning the conservation of homological properties of commutative noetherian rings under the tensor product operation involves a flatness assumption. In that sense, the above Corollary is different, because, to our knowledge, it is not known if in the above situation the maps $A\to \widehat{A\otimes_{\k} B}$ and $B\to \widehat{A\otimes_{\k} B}$ are flat (because it is not known if the completion map $A\otimes_{\k} B \to \widehat{A\otimes_{\k} B}$ is flat when $A\otimes_{\k} B$ is non-noetherian),  although, flatness is known to hold if $A/\a$ is essentially of finite type over $\k$ (See \cite[Proposition 7.1(b)]{AJL2}). We thus view this result as another example of the fact that weak proregularity of the ideal of definition of the adic topology can serve as a replacement for flatness of the completion map in many interesting situations.
\end{rem}

Cohen structure theorem may be stated as follows: given a noetherian local ring $(A,\mfrak{m})$, its completion $\Lambda_{\mfrak{m}}(A)$ is a quotient of a regular local ring. Our final corollary is a weak variation of it for tensor product of local rings. It says that the completion of a tensor product of local rings is a quotient of a Gorenstein ring.
\begin{cor}
Let $\k$ be a field, and let $(A,\mfrak{m})$ and $(B,\mfrak{n})$ be two noetherian local $\k$-algebras. Let $I = \mfrak{m}\otimes_{\k} B + A\otimes_{\k} \mfrak{n}$, and assume that 
\[
\widehat{A\otimes_{\k} B} := \Lambda_I(A\otimes_{\k} B)
\]
is noetherian of finite Krull dimension. Then $\widehat{A\otimes_{\k} B}$ has dualizing complexes, so it is a quotient of a Gorenstein ring of finite Krull dimension.
\end{cor}
\begin{proof}
By Cohen structure theorem, the rings $\widehat{A}$ and $\widehat{B}$ have dualizing complexes. Since there are isomorphisms $A/{\mfrak{m}}^n \cong \widehat{A}/ (\mfrak{m}\cdot \widehat{A})^n$ and $B/{\mfrak{n}}^n \cong \widehat{B}/ (\mfrak{n}\cdot \widehat{B})^n$, we see as in the proof of Lemma \ref{lem:reduction-to-discrete} that
\[
\widehat{A\otimes_{\k} B} \cong \varprojlim (A/{\mfrak{m}}^n \otimes_{\k} B/{\mfrak{n}}^n) \cong 
\varprojlim (\widehat{A}/ (\mfrak{m}\cdot \widehat{A})^n \otimes_{\k} \widehat{B}/ (\mfrak{n}\cdot \widehat{B})^n ) \cong \widehat{\widehat{A}\otimes_{\k} \widehat{B}},
\]
where $\widehat{\widehat{A}\otimes_{\k} \widehat{B}} := \Lambda_{J}(\widehat{A}\otimes_{\k} \widehat{B})$,
$J := (\mfrak{m}\cdot\widehat{A})\otimes_{\k} \widehat{B} + \widehat{A}\otimes_{\k} (\mfrak{n}\cdot \widehat{B})$. By Theorem \ref{thm:tensor-dual-adic}, the ring $\widehat{\widehat{A}\otimes_{\k} \widehat{B}}$ has dualizing complexes, 
so the isomorphic ring $\widehat{A\otimes_{\k} B}$ also has dualizing complexes. Hence, by Kawasaki's theorem, it is a quotient of a Gorenstein ring of finite Krull dimension.
\end{proof}

\end{document}